\documentclass{article}
\usepackage{ijcai20}
\usepackage[export]{adjustbox}
\usepackage{times}

\usepackage{soul}
\usepackage{url}
\usepackage[hidelinks]{hyperref}
\usepackage[utf8]{inputenc}
\usepackage[small]{caption}
\usepackage{subcaption}
\usepackage{graphicx}
\usepackage{amsmath}
\usepackage{booktabs}
\usepackage{graphicx}

\usepackage[ruled,linesnumbered,noresetcount,vlined]{algorithm2e}

\SetKwFor{Repeat}{repeat}{:}{endw}
\usepackage{amsmath,amsfonts,mathrsfs,amssymb,bm,mathtools,mathabx}
\usepackage{graphicx, epsfig, epstopdf}
\usepackage{wrapfig, multirow, multicol,mdframed}
\usepackage{rotating, booktabs}
\usepackage{stmaryrd,xspace,arydshln}
\usepackage{picinpar}
\usepackage{todonotes}
\usepackage{color}
\usepackage{enumitem}
\usepackage{cleveref}

\usepackage{float}
\floatstyle{ruled}
\newfloat{model}{thp}{lop}
\floatname{model}{Model}

\newcommand{\citep}[1]{\citeauthor{#1}~[\citeyear{#1}]}

\newcommand{\bitemize}{\begin{list}{$\bullet$}{\topsep=1pt \parsep=0pt \itemsep=1pt \leftmargin=1em }} 
\newcommand{\eitemize}{\end{list}}
\newcommand{\benumerate}{\hspace{-0.7in}\begin{enumerate}\topsep=0pt \parsep=0pt \itemsep=-3pt}
\newcommand{\eenumerate}{\end{enumerate}}
\newcommand{\beitemize}{\begin{list}{$\bullet$}{\topsep=1.5pt \parsep=0pt \itemsep=1pt \leftmargin=1em }} 
\newcommand{\enitemize}{\end{list}}

\def\thmspace{0.2em}
\newtheorem{theorem}{\hspace{\thmspace}{\bf Theorem}\!}

\newenvironment{proof}{{\textit{Proof}.}}{\hfill$\Box$}

\def \iff{\;\Leftrightarrow\;}

\newcommand{\D}{{\mathscr{D}}}  
\newcommand{\R}{{\mathscr{R}}}  

\newcommand{\cF}{\mathcal{F}}

\newcommand{\cM}{\mathcal{M}}

\newcommand{\cO}{\mathcal{O}}
\newcommand{\cP}{\mathcal{P}} %

\newcommand{\bx}{\bm{x}}
\newcommand{\by}{\bm{y}}

\newcommand{\EE}{\mathbb{E}}
\newcommand{\RR}{\mathbb{R}}

\newcommand{\Lap}{{\text{Lap}}}

\SetAlgoSkip{}
\SetAlgoInsideSkip{}
\title{Differential Privacy for Stackelberg Games}

\author{
Ferdinando Fioretto$^{1}$\footnote{Authors names listed alphabetically. All authors have equal contributions.}\and
Lesia Mitridati$^{2}$\And
Pascal Van Hentenryck$^{2}$\\
\affiliations
$^1$Syracuse University, 
$^2$Georgia Institute of Technology\\
\emails
ffiorett@syr.edu,
lmitridati3@gatech.edu,
pvh@isye.gatech.edu
}

\begin{document}

\maketitle\sloppy\allowdisplaybreaks

\begin{abstract}  
This paper introduces a \emph{differentially private} (DP) mechanism to protect the information exchanged during the coordination of \textit{sequential} and \textit{interdependent} markets. This coordination represents a classic Stackelberg game and relies on the exchange of \emph{sensitive information} between the system agents. 
The paper is motivated by the observation that the perturbation introduced by traditional DP mechanisms fundamentally changes the underlying optimization problem and even leads to unsatisfiable instances. 
To remedy such limitation, the paper introduces the \emph{Privacy-Preserving Stackelberg Mechanism} (PPSM), a framework that enforces the notions of \textit{feasibility} and \textit{fidelity} of the privacy-preserving information to the original problem objective. PPSM complies with the notion of differential privacy and ensures that the outcomes of the privacy-preserving coordination mechanism are close-to-optimality for each agent. 
Experimental results on several gas and electricity market benchmarks based on a real case study demonstrate the effectiveness of the approach.
\end{abstract}

\section{Introduction}

%

The coordination of \textit{sequential} and \textit{interdependent}
agents in market-based systems has traditionally been modeled as a
\textit{Stackelberg game} \cite{simaan1973stackelberg}.  While the
coordination among these agents is central to achieve efficient and
cost-effective operations, it also requires the exchange of
proprietary information between the agents in order to achieve an
optimal strategy.  For instance, in the context of electricity and gas
markets, relevant data may represent the costs of producers, the loads
of consumers, or technical characteristics of the energy network.  As
has been observed in various works (e.g.,
\cite{zugno2013pool,baringo2013strategic}), such information may be
sensitive: It can provide a competitive advantage over other strategic
agents in the system, it may induce financial losses, and it may even
benefit external attackers \cite{maharjan2013dependable}.

To address this issue, several privacy-preserving frameworks have been
proposed, with \emph{Differential Privacy} (DP) \cite{dwork:06}
emerging as a robust privacy framework for many applications.  DP
allows to measure and bound the risk associated with an individual
participation in an analysis task. DP algorithms rely on the injection
of carefully calibrated noise to the output of a computation. They can
thus be used to \emph{obfuscate} the sensitive data exchanged by the
system agents in the market. However, as shown in Section
\ref{sec:experiments}, when perturbed data are used as input to
\emph{Stackelberg games}, they may produce results that are
fundamentally different from those obtained on the original data: They
often transform the nature of the underlying optimization problem and
even lead to \emph{severe feasibility issues}.

This paper is a first step in addressing this challenge. It introduces
the \emph{Privacy-Preserving Stackelberg Mechanism} (PPSM) for the
coordination of sequential and interdependent agents. PPSM is a
two-stage protocol that allows the coordinating agents to exchange
differentially private data of high fidelity. In particular, PPSM
relies on several learning and optimization models (including
\emph{fidelity constraints} on the objectives and coordination
variables of the agents) to ensure that the privacy-preserving
information exchanged between the agents preserves the feasibility and
near-optimality of the original Stackelberg game. PPSM has been
analyzed both theoretically and experimentally. The theoretical
guarantees ensure differential privacy and near optimality, while the
experimental results validate the approach on a real test case for the
coordination of electricity and natural gas markets in the
Northeastern United States \cite{byeon2019unit}. The case study shows that PPSM can bring up to two orders of magnitude error reduction over standard privacy-preserving mechanisms.

Although the paper was motivated by the coordination between natural
gas and electricity markets, the proposed methods may apply to any type of coordination mechanism between sequential and interdependent agents where the agents exchange information and synchronize through price signals. 

\section{Problem Definition and Privacy Goal} 
\label{sec:preliminaries}

\begin{figure}[!t]
    \centering
    \includegraphics[width=.9\linewidth]{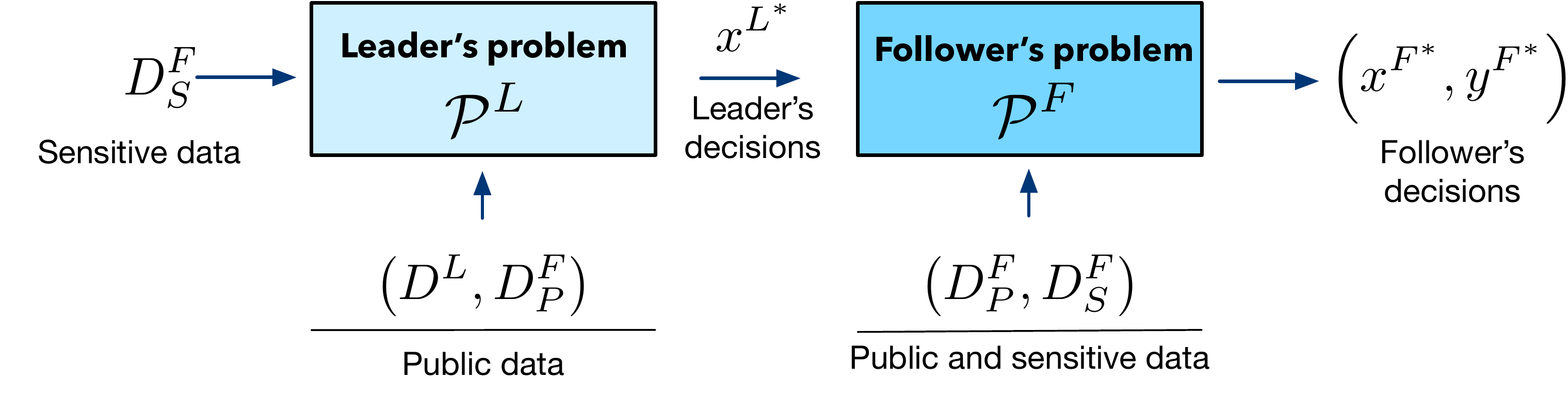}
    \caption{Stackelberg game in sequential interdependent markets.}
    \label{fig:S}
\end{figure}

The strategies of two sequential and interdependent agents, such as
energy market operators, represent a classic
\emph{Stackelberg game} \cite{simaan1973stackelberg}. In this
framework, which is schematically illustrated in Figure \ref{fig:S},
the \emph{leader} (e.g., the first market operator) optimizes its
decisions while anticipating the reaction of the \emph{follower}
(e.g., the second market operator). The leader actions impact the
reaction of the follower, which in turn impacts the leader objective
value. As a result, the leader strategy in Stackelberg games can be
modeled as a bilevel optimization problem $\cP^L\left(D^L,D^F_P,D^F_S\right)$:
\begin{subequations} \label{eq:hierarchical}
{\small
\begin{alignat}{2}
  {\cO}^{{L}^*}=\min_{\bx^L\!\!,\by^F} &
  	\cO^L 
  		\left(\bx^L\!\!\!, D^L\right) 
  		\label{eq:hierarchical0} \\
  \text{s.t.}\;\; & 
  (\bx^L\!\!\!, \by^F)  \in \cF^L(D^L) 
  \label{eq:hierarchical1} \\
  & 
  \by^F\!\!\!= \text{dual sol. } \min_{\bx^F} \; \cO^F \left(\bx^F\!\!\!, D^F_P, D^F_S\right) 
   \label{eq:hierarchical2} \\ 
   	& \hspace{56pt} \text{s.t. } \; 
    (\bx^F\!\!\!, \bx^L)  \in \cF^F (D^F_P, D^F_S), 
	\label{eq:hierarchical3}
\end{alignat}
}
\end{subequations}
\noindent where $\bx^L$ represents the vector of decision variables of the
leader, and $\bx^F$ and $\by^F$ the vectors of \emph{primal} and
\emph{dual} decision variables of the follower. Additionally, $D^L$ and ($D^F_P, D^F_S)$ are the inputs of the leader and follower problems, respectively. The follower inputs are either \textit{public}
($D^F_P$) or \textit{sensitive} ($D^F_S$). The upper-level problem
minimizes the leader objective cost \eqref{eq:hierarchical0},
constrained by its feasible decision space $\cF^L$
\eqref{eq:hierarchical1}, and the reaction of the follower in the
lower-level problem \eqref{eq:hierarchical2} and
\eqref{eq:hierarchical3}.  The follower problem, denoted by
$\cP^F\left(\bx^L,D^F_P,D^F_S\right)$, minimizes the follower
objective cost \eqref{eq:hierarchical2}, constrained by the feasible space $\cF^F$ of the follower decisions \eqref{eq:hierarchical3}.

\smallskip\noindent{\textbf{Coordination Variables}:} The leader
\emph{primal} variables $\bx^L$ appear as fixed parameters in the
expression of the follower feasible space $\cF^F\!$. In return, the
lower-level problem provides feedback from the follower \emph{dual}
variables $\by^F$ to the upper-level problem through its feasible
space $\cF^L$. In energy markets, primal variables typically represent
commitment and energy production decisions, while dual variables
represent energy prices.  These variables shared between the follower and the leader are called \textit{coordination variables}.

\smallskip\noindent{\textbf{Assumptions}:} Due to the sequential
decision-making nature, the leader needs to anticipate the reaction of
the follower.  Therefore, this paper assumes that the leader has
access to a prediction model $\cM^L(D^L,D^F_P,D^F_S)$ that predicts
the values $\bar{y}^F$ of the follower dual variables. This is a
natural assumption in energy markets applications that motivates this
paper: Such forecasting models are used in practice since generators
must predict energy prices in order to efficiently bid in the markets and the market needs to ensure reliability of the overall
system. Similarly, the follower has access to a forecasting model
$\cM^F(\bx^L, D^F_P, D^F_S)$ that predicts its objective value
$\bar{\cO}^{F^*}\!\!$ and the values of its dual variables
$\bar{y}^{F^*}\!\!$.  This is also a natural assumption in energy markets
since the energy prices and costs, representing the values of the dual
variables and objective cost of the follower problem
$\cP^F\left(\bx^L,D^F_P,D^F_S\right)$, are released publicly, and thus
can be used to train precise estimators. The PPSM mechanism in this
paper applies these forecasting models on privacy-preserving
versions of the sensitive parameters.

\smallskip\noindent{\textbf{Motivation Problem}:}
\citep{byeon2019unit} recently showed that the coordination between
electricity and natural gas markets can be modeled as a Stackelberg
game between a leader, i.e. the \emph{gas-aware electricity unit
  commitment} (GAUC), and two followers, i.e. the \emph{electricity
  market} (EM) and \emph{natural gas market} (GM). This game can
alleviate the reliability issues that emerged in the recent polar
vortex events. In this context, the leader coordination variables
represent the \textit{commitment} of Gas-Fired Power Plants (GFPPs),
which impacts their participation in both EM and GM. The relevant
follower coordination variables represent natural gas \textit{prices} in the GM. These prices impact the GAUC decisions through
\textit{coordination constraints} representing the profitability of
the bids of GFPPs.

\smallskip\noindent{\textbf{Privacy Goal}:} The paper focuses on situations where the follower inputs $D^F_S$
contain sensitive information that should not be revealed.  In the
case of electricity and natural gas markets, a FERC directive allows the gas and electricity operators to share network data, while the bids of the generators are public information. Therefore, the sensitive parameters typically represent the gas demand profile of consumers. As discussed in the introduction, if released, they can provide a competitive advantage to strategic agents in the energy system and may result in financial losses for the follower, as shown
in \cite{zugno2013pool,baringo2013strategic}.  Thus, the \emph{privacy goal} is to ensure that the sensitive information $D^F_S$ is not
breached during the coordination process described in Problem
\eqref{eq:hierarchical}. The next section introduces a formal notion that will be used to achieve this goal.

\section{Background: Differential Privacy}
\label{sec:differential_privacy}

\emph{Differential privacy} \cite{dwork:06} (DP) is a rigorous privacy notion 
used to protect disclosures of an individual's data in a computation. The paper 
considers datasets  $D^F_S \!=\! (d_1,\ldots,d_n)$ 
with each $d_i \!\in\! \RR_+$ describing a sensitive quantity, such as the participants' demand value in the GM. 
DP relies on the notion of dataset adjacency, which captures the differential information to be protected. To protect the values in the dataset, the paper uses the following \emph{adjacency relation}:
\begin{align*}
  D \sim_\alpha D' \iff \exists i: |d_i - d_i'| \leq \alpha \land\; 
  \forall j \neq i: d_j = d_j',
\end{align*}
where $D$ and $D'$ are two datasets and $\alpha$ is a positive real value. Such definition is useful to hide individual participation up to some quantity $\alpha$ \cite{chatzikokolakis2013broadening}. In the paper context, it allows customers to reveal gas demand profiles that hide the real consumption by a factor of $\alpha$.

A randomized mechanism $\cM \!:\! \D \!\to\! \R$ with
domain $\D$ and range $\R$ is $\epsilon$-DP if, for any output response $O \subseteq \R$ and any two \emph{adjacent}
datasets $D \sim_\alpha D'$ , for a fixed value $\alpha > 0$: 
\begin{equation}
  \label{eq:dp_def}
  Pr[\cM(D) \in O] \leq \exp(\epsilon)\, Pr[\cM(D') \in O].
\end{equation}
\noindent 
The parameter $\epsilon \geq 0$ is the \emph{privacy loss} of the
mechanism, with values close to $0$ denoting strong privacy. The level of \emph{indistinguishability} is controlled by the parameter $\alpha > 0$. The definition above was introduced by \citep{chatzikokolakis2013broadening}. It is a \emph{generalization} of the classical DP definition, that protects individuals participation, to arbitrary metric spaces. W.l.o.g.~the paper fixes $\epsilon=1.0$ and focuses in the indistinguishability parameter $\alpha$. In the context of this work, Definition \eqref{eq:dp_def} is referred to as $\alpha$-indistinguishability.

An important DP property is immunity to post-processing.
\begin{theorem}[Post-Processing \cite{dwork:13}] 
	\label{th:postprocessing} 
	Let $\cM$ be an $\alpha$-indistinguishable mechanism and $g$ be an arbitrary mapping from the set of possible outputs to an arbitrary set. Then, $g \circ \cM$ \mbox{is $\alpha$-indistinguishable.}
\end{theorem}

A function $Q$ (also called \emph{query}) from a data set $D \in \D$ to a result set $R \subseteq \RR^n$ 
can be made differentially private by injecting random noise to its
output. The amount of noise depends on the \emph{sensitivity} of the
query, denoted by $\Delta_Q$ and defined as
\(
\Delta_Q = \max_{D \sim D'} \left\| Q(D) - Q(D')\right\|_1.
\)
In other words, the sensitivity of a query is the maximum
$l_1$-distance between the query outputs of any two adjacent
datasets $D$ and $D'$.

\section{The PPS Problem}
\label{sec:PPS_problem}

The Privacy-Preserving Stackelberg (PPS) problem establishes the
fundamental desiderata to be delivered by the obfuscation
mechanism. It operates on the follower sensitive parameters $D_F^S$
exchanged to ensure coordination between the leader and the follower
in the resolution of Problem \eqref{eq:hierarchical}. 
The goal is to produce a privacy-preserving version $\hat{D}_F^S$ 
that is $\alpha$-indistinguishable from $D_F^S$ and which ensures that
$\hat{\cO}^{{L}^*} \approx {\cO}^{{L}^*}$, where $\hat{\cO}^{{L}^*}$ 
is the leader objective value in the Stackelberg game 
when $D_F^S$ is replaced by $\hat{D}_F^S$.

\section{The PPSM Mechanism}
\label{sec:PPSM}

The \emph{Privacy-Preserving Stackelberg Mechanism} (PPSM) is
described schematically in \Cref{fig:PPSM} and consists of the
following steps which will be described in detail subsequently:

\begin{itemize}[noitemsep,nolistsep,leftmargin=6pt,itemindent=6pt]

\item {\bf [1] [Follower]:} Given the sensitive data $D^F_S$, the
  follower produces $\tilde{D}^F_S$ which is
  $\alpha$-indistinguishable from $D^F_S$.
    
\item {\bf [2a] [Leader]:} Given $\tilde{D}^F_S$ and $D^L$, 
the leader uses model $\cM^L$ to (privately) estimate the value 
of the dual coordination variables $\bar{y}^F$.

\item {\bf [2b] [Leader]:} It then solves the leader problem 
$\cP^L(D^L, \bar{y}^L)$ with the values of variables $\by^L$ 
fixed to $\bar{y}^L$ to obtain the value of the coordination 
variables $\bar{x}^L$.
   
\item {\bf [3a] [Follower]:} Given values $\bar{x}^L$ and the 
follower data $(D^F_P, \tilde{D}^F_S)$, the follower uses model 
$\cM^F$ to (privately) estimate the objective value 
$\bar{\cO}^{F* }$ and the values for the dual variables 
$\bar{y}^{F*}$ of the follower problem $\cP^F$.

\item {\bf [3b] [Follower]:} Finally, using $\tilde{D}^F_S$ and the estimates computed in [2b] and [3a], the follower produces a new privacy-preserving vector $\hat{D}^F_S$ to achieve $\hat{\cO}^{{L}^*} \approx {\cO}^{{L}^*}$.
\end{itemize}

\noindent Once PPSM produces the privacy-preserving demand
$\hat{D}^F_S$, the leader can clear its market using $\hat{D}^F_S$ to produce $\bar{x}^{L*}$ which is communicated to the follower for clearing its own market as depicted in Figure \ref{fig:S}.

\begin{figure}[!t]
    \centering
    \includegraphics[width=1.0\linewidth]{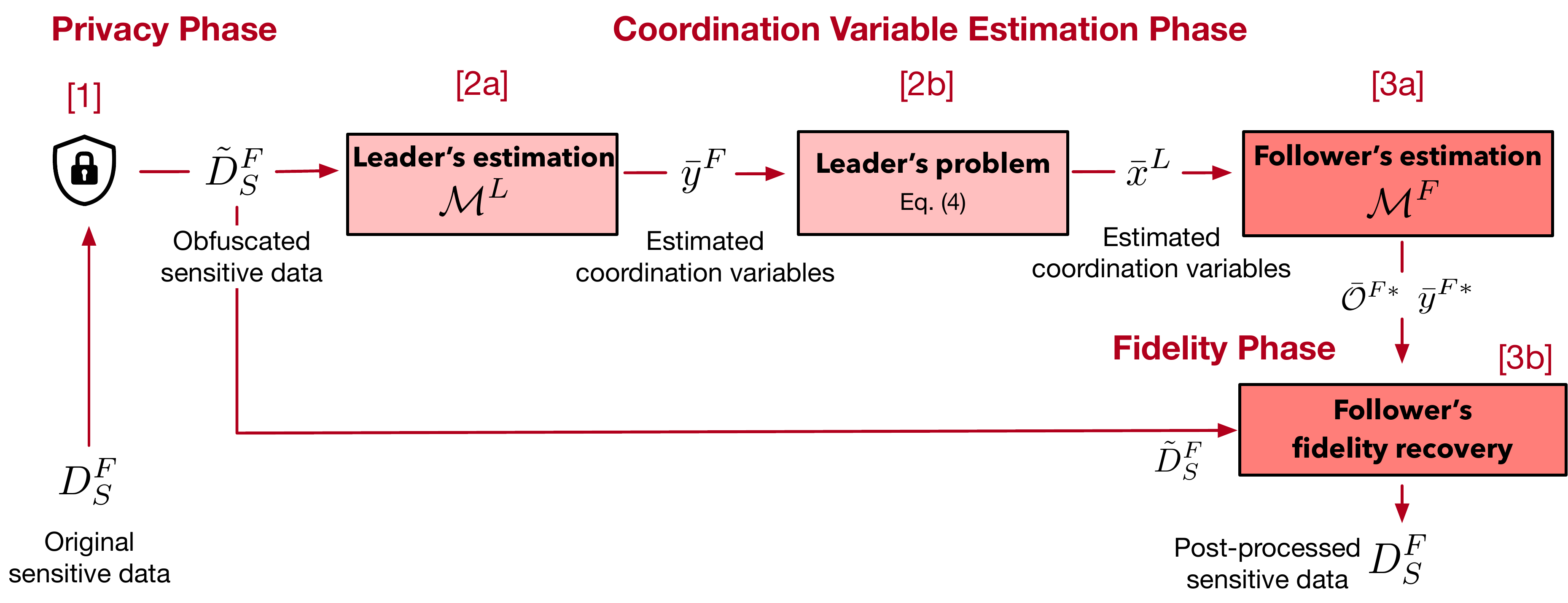}
    \caption{PPSM Illustration.}
    \label{fig:PPSM}
\end{figure}

\subsection{Privacy Phase}

In Step [1], the follower takes as input its sensitive parameters  $D^F_S$ and constructs a privacy-preserving version $\tilde{D}^F_S$ using the Laplace mechanism \cite{dwork:06}.
\begin{theorem}[Laplace Mechanism]
\label{th:m_lap} 
Let $Q$ be a numeric query that maps datasets to $\RR^n$. The Laplace mechanism that outputs $Q(D) + \xi$, where $\xi \in \RR^n$ is drawn from the Laplace distribution $\Lap(\Delta_Q / \epsilon)^n$, achieves $\alpha$-indistinguishability.
\end{theorem}

\noindent
In the above, 
$\Lap(\lambda)^n$ denotes the i.i.d.~Laplace distribution with 0
mean and scale $\lambda$ over $n$ dimensions. 
As a result, the privacy-preserving parameters $\tilde{D}^F_S$ are obtained as follows:
\begin{equation} \label{eq:setp1}
	\tilde{D}^F_S  = D^F_S + \Lap(\alpha)^n.
\end{equation}
Importantly, the Laplace mechanism has been shown to be \emph{optimal}, i.e., it minimizes the mean-squared error for identity queries w.r.t.~the L1-norm \cite{koufogiannis:15}.

While \eqref{eq:setp1} ensures $\alpha$-indistinguishability, the
obfuscated data may not achieve strong \emph{fidelity} with respect to the
original problem. Crucially, in energy systems, the inputs generated
by this mechanism often fail to produce a feasible solution to the
problem of interest, as highlighted in Section
\ref{sec:experiments}. To remedy this limitation, the proposed PPSM introduces an \emph{optimization-based} approach that aims at producing a new privacy-preserving dataset $\hat{D}^F_S$ establishing \emph{fidelity} with respect to the constraints and objectives of the leader and follower problems.

\subsection{Estimating the Coordination Variables}

After having received the Laplace-obfuscated data $\tilde{D}^F_S$ 
from the follower, the leader uses model $\cM^L(D^L,D^F_P, \tilde{D}^F_S)$ 
to estimate the value of the coordination variables $\by^F$ of the follower 
(Step [2a]). 
Next, in Step [2b], the leader solves the optimization problem 
\begin{equation} \label{eq:step_2b}
\bar{x}^L = \arg\min_{\bx^L} \cO^L \left(\bx^L\!\!\!, D^L\right) 
\;\;\text{s.t.}\;\; (\bx^L\!\!\!, \bar{y}^{F})  \in \cF^L(D^L). 
\end{equation}
Problem \eqref{eq:step_2b} describes, in fact, the leader \emph{subproblem}
(\Cref{eq:hierarchical0,eq:hierarchical1}), in which the values 
of variables $\by^F$ have been fixed to the estimates $\bar{y}^{F}$. 

The leader then communicates the estimates $\bar{x}^L$ to the
follower. In turn, the follower uses model $\cM^F(\bar{x}^L, D^F_P, 
\tilde{D}^F_S)$ to estimate the values of its subproblem objective 
value $\bar{\cO}^{F^*}$ and dual variables $\bar{y}^{F^*}$ (Step [3a]).


\subsection{Fidelity Phase} 
\label{sec:fidelity_phase}

Finally, given the obfuscated parameters $\tilde{D}^F_S$, computed in Step [1], the estimated values $\bar{x}^L$, computed in Step [2b], and the follower objective value $\bar{\cO}^{F^*}$ and dual variables $\bar{y}^{F^*}$, computed in Step [3a], the follower executes the following bilevel optimization problem:

\begin{subequations} \label{eq:pp3}
\begin{align}
  \min_{\hat{D}^F_S, \hat{\bx}^F, \hat{\by}^F} \;\;& 
  \| \hat{D}^F_S - \tilde{D}^F_S\|^2_2 \label{eq:pp3_0} \\
  \text{s.t.}\;\;& 
  |  \hat{\cO}^{{F}} - \bar{\cO}^{F^*} | \leq \eta_p \label{eq:pp3_1a} \\
  & |  \hat{\by}^F - \bar{y}^{F^*} | \leq \eta_d \label{eq:pp3_1b} \\
& \hat{\by}^F = \text{dual sol. of } \cP^F(\bar{x}^L, \!D^F_P, \!\hat{D}^F_S), \label{eq:pp3_2}
\end{align}
\end{subequations}
\noindent
where $\eta_p$ and $\eta_d$ are parameters specifying the desired
fidelity for the value of the objective and dual variables of the 
follower.  Its objective is to find new values $\hat{D}^F_S$ that minimize the
distance to the Laplace-obfuscated $\tilde{D}^F_S$ \eqref{eq:pp3_0},
while ensuring (component-wise) fidelity with respect to the estimated
objective value $\bar{\cO}^{F^*}\!\!$ \eqref{eq:pp3_1a} and dual
variables $\bar{y}^{F^*}\!\!$ \eqref{eq:pp3_1b}. The follower
objective function ${\hat{\cO}}^{F}$ and dual variables
${\hat{y}}^{F}$ are defined as the solutions to the lower-level
problem \eqref{eq:pp3_2}, which represents the follower subproblem
$\cP^F(\bar{x}^L,D^F_P,\hat{D}^F_S)$ with the new sensitive parameters
$\hat{D}^F_S$ and the values of the coordination variables $\bx^L$ fixed to the estimates $\bar{x}^L$. Additionally, this
lower-level problem \eqref{eq:pp3_2} enforces feasibility of the follower problem with respect to the estimates $\bar{x}^L$.

Using the equivalent Karush-Kuhn-Tucker (KKT) conditions of the linear
lower-level problem \eqref{eq:pp3_2} and the Fortuny-Amat
linearization, this bilevel problem can be recast as a mixed-integer
second-order cone program (MISOCP) \cite{gabriel2012complementarity}.




\begin{theorem}
For given positive real $\alpha$, PPSM satisfies $\alpha$-indistinguishability.
\end{theorem}
\begin{proof}
The privacy phase of PPSM (Step [1]) produces a gas demand profile $\tilde{D}^F_S$ using the Laplace mechanism with parameter $\lambda = \alpha$ (\Cref{eq:setp1}). By \Cref{th:m_lap}, the resulting demand profile satisfies $\alpha$-indistinguishability. 
Notice that Steps [2a]--[3b] take as input the privacy-preserving gas profiles $\tilde{D}^F_S$ generated during the privacy phase, and uses exclusively the public information---e.g., the problem model and the historical data.  
Therefore, by the post-processing immunity property of differential privacy (\Cref{th:postprocessing}) the output $\hat{D}^F_S$ of the fidelity phase satisfies $\alpha$-indistinguishability.
\end{proof}

\section{Error Analysis}
\label{sec:theo}

This section analyzes the impact of the data perturbation induced by $\text{PPSM}$ on the optimal objective value of both agents in the privacy-preserving problem. 
The results below hold under the assumption that the estimated values $\bar{x}^{L}$, $\bar{\cO}^{F^*}$, and $\bar{y}^{F^*}$ are accurate. 

\begin{theorem}[Error]
\label{thm:error}
  After the fidelity phase, the expected error induced by $\text{PPSM}$ on  the original, sensitive, parameters $D^F_S$ is bounded by the inequality: 
  $$
    \EE[ \| \hat{D}^{F^*}_S - D^F_S \| ] \leq 
    4 \alpha^2, 
  $$
  where $\hat{D}^{F^*}_S$ is the solution to Problem \eqref{eq:pp3}. 
\end{theorem}
\begin{proof}
  Denote with $\tilde{D}^F_S = D^F_S + \Lap(\alpha)$ the privacy-preserving version of the gas demand profile obtained by the Laplace mechanism. We have that:
  \begin{align*}
           \| \hat{D}^{F^*}_S - D^F_S \|_{2} 
    & \leq \| \hat{D}^{F^*}_S - \tilde{D}^F_S \|_{2} + 
           \| \tilde{D}^F_S - D^F_S \|_{2} \label{eq:p2}\\
   &\leq 2 \|\tilde{D}^F_S - D^F_S\|_{2} \leq 4 \alpha^2.
  \end{align*}
  The first inequality follows from the triangle inequality on norms. 
  The second inequality follows from:
  $$
      \| \hat{D}^{F^*}_S - \tilde{D}^F_S \|_{2} \leq \|\tilde{D}^F_S - D^F_S\|_{2}
  $$
  by optimality of 
  $\langle \hat{D}^{F^*}_S, \hat{\bx}^*\rangle$ 
  and the fact that 
  $\langle D^F_S, \bx^*\rangle$ is a feasible solution to constraints 
  \eqref{eq:pp3_1a} to \eqref{eq:pp3_2} (Problem \ref{eq:pp3}), 
  or, similarly, by optimality of
  $\langle \hat{D}^{F^*}_s, \hat{\bx}^*, \hat{\by}^*\rangle$ 
  and the fact that 
  $\langle D^F_S, \bx^*, \by^*\rangle$ is a feasible solution to constraints 
  \eqref{eq:pp3_1a} to \eqref{eq:pp3_2} (Problem \ref{eq:pp3}).
  The third inequality follows directly from the variance of the Laplace distribution.
\end{proof}


The next theorem bounds the difference in objective value for the
leader problem when the leader problem is convex and the coordination constraints are linear.

\begin{theorem}[Cost of privacy]
\label{thm:fidelity}
Consider a \emph{convex} leader problem whose coordination constraints \eqref{eq:coordination} 
are \emph{linear};
\begin{subequations} \label{eq:convex_leader}
\begin{alignat}{2}
\textstyle \hat{{\cal O}}^{L^*} = & \min_{\bx^L,\by^F} && {\cal O}^{L}\left(\bx^L,D^L\right) \\
& \text{s.t.} && \bx^L \in \mathcal{F}^L\left(D^L\right) \\
& \quad && A \bx^L + B \by^F \geq b \label{eq:coordination} \\
& \quad && \by^F = \text{dual sol. of } \cP^F(\bx^L, \!D^F_P, \!\hat{D}^F_S),
\end{alignat}
\end{subequations}
where $\cP^F(\bx^L, \!D^F_P, \!\hat{D}^F_S)$ uses the privacy-perserving
$\hat{D}^F_S$.  

\noindent After the fidelity phase, the error induced by
$\text{PPSM}$ on the objective value $\hat{\cO}^{L^*}$ of the leader
problem \eqref{eq:convex_leader} is bounded by the inequality:
\begin{equation} \label{eq:bound}
    \| \hat{\cO}^{L^*} - \cO^{L^*} \| 
      \leq \eta_d \|B^T y^{L^*}\|_1, 
\end{equation}
where $y^{L^*}$ denotes the dual variables values associated with the constraints \eqref{eq:coordination} of the original leader problem $\cP^L(D^L,D^F_P,D^F_S)$.
\end{theorem}

\begin{proof}
Equation \eqref{eq:pp3_1b} enforces fidelity w.r.t.~the dual coordination variables values $\bar{y}^{F^*}$, as estimated in step [3a]. Therefore, the upper- and lower-level variables of the DP leader problem \eqref{eq:convex_leader}, denoted $\cP^L(D^L,D^F_P,hat{D}^F_S)$, can be defined as the perturbed variables: $\by^F= \bar{y}^{F^*} + \eta_d \bm{\epsilon}$ and $\bx^L = x^{L^*} + \bm{\delta x^L}$; where, $x^{L^*}$ are the solutions to the leader problem $\cP^L(D^L,D^F_P,D^F_S)$ with original sensitive parameters $D^F_S$, and the variables $\bm{\epsilon} \in \mathcal{U}$ such that $\mathcal{U} \subset \mathbb{R}^J = \left\{ \left[\epsilon_1,\ldots,\epsilon_J\right] \mid -1 \leq \epsilon_j \leq 1, \ \forall j \in \{1,\ldots,J\} \right\} $ ($J$ is the dimensionality of the dual variables $\by^F$)).

As a result, with a change of variables in Problem \eqref{eq:convex_leader}, the difference between the objective value $\hat{\cO}^{L^*}$ of the DP leader problem $\cP^L(D^L,D^F_P,\hat{D}^F_S)$ and the objective value $\cO^{L^*}$ of the original leader problem $\cP^L(D^L,D^F_P,D^F_S)$ can be expressed as the optimal objective value of the following \emph{perturbed} leader problem:
\begin{subequations} \label{eq:perturbed_leader}
\begin{align}
& \hat{\cO}^{L^*} - \cO^{L^*} 
= \min_{\bm{\delta x^L},\bm{\epsilon}} \quad \cO^{L}\left(\bm{\delta x^L},D^L\right)\\
&  \text{s.t} \quad x^{L^*} + \bm{\delta x^L} \in \mathcal{F}^L\left(D^L\right) \\ 
 & \quad \quad A \bm{\delta x^L} + B \eta_d \bm{\epsilon} \geq b - A x^{L^*} - B \bar{y}^{F^*} \\
 & \quad \quad   \bar{y}^{F^*} + \eta_d \bm{\epsilon} = \text{dual sol. of } \cP^F(\bx^L,D^F_P,\hat{D}^F_S) \\
 & \quad \quad \quad \quad \quad \quad \quad \quad \text{problem } \eqref{eq:hierarchical2}-\eqref{eq:hierarchical3}
 \end{align}
\end{subequations}
In this perturbed problem, the coordination constraints \eqref{eq:coordination}, are perturbed by the term $\eta_d \bm{\epsilon}$.

We define the following function $f: \bm{\epsilon} \in \mathcal{U} \mapsto f\left(\bm{\epsilon}\right)  \in \mathbb{R} $, as the optimal objective value of the following (parametric) optimization problem:
\begin{subequations} \label{eq:uncertain_function}
\begin{align}
& f\left(\bm{\epsilon}\right)
= \min_{\bm{\delta x^L}} \quad  \cO^{L}\left(\bm{\delta x^L},D^L\right)  \\ 
& \text{s.t} \quad  x^{L^*} + \bm{\delta x^L} \in \mathcal{F}^L\left(D^L\right)  \\
& \quad \quad  A \bm{\delta x^L} + B \eta_d \bm{\epsilon} \geq b - A x^{L^*} - B \bar{y}^{F^*}, \\
& \forall \bm{\epsilon} \in \mathcal{U}.
\end{align}
\end{subequations}
The bounds of this function over the set $\bm{\epsilon} \in \mathcal{U}$ will provide bounds to the objective value of the perturbed Problem \eqref{eq:perturbed_leader}.
For any $\bm{\epsilon} \in \mathcal{U}$, $f\left(\bm{\epsilon}\right)$ is upper-bounded by the optimal objective value of the following problem:
\begin{subequations} \label{eq:wc_leader}
\begin{align}
&  \min_{\bm{\delta x^L}} \quad \cO^{L}\left(\bm{\delta x^L},D^L\right)  \\  
& \text{s.t} \quad  x^{L^*} + \bm{\delta x^L} \in \mathcal{F}^L\left(D^L\right)  \\
&  \quad  \sum_{k=1}^{K} A_{iK} \bm{\delta x_k^L} + \min_{\bm{\epsilon} \in \mathcal{U}} \sum_{j=1}^{J} B_{ij} \eta_d \bm{\epsilon_j} \nonumber \\
& \quad \ \geq b_i - \sum_{k=1}^{K} A_{ik} x_k^{L^*} - \sum_{j=1}^{J} B_{ij} \bar{y}_j^{F^*} ,  \ \forall i \in \{1,...,I\}. \label{eq:uncertain_function_bounds1}
\end{align}
\end{subequations}
Note that, \label{eq:uncertain_function_bounds1} represents a \emph{worst-case} scenario of the perturbed coordination constraints \eqref{eq:coordination}.

The solutions to the inner minimization in equation \eqref{eq:uncertain_function_bounds1} is $-\sum_{j=1}^{J} | B_{ij}| \eta_d$. Therefore, in the worst-case leader problem \eqref{eq:wc_leader}, the right-hand side of each coordination constraint \eqref{eq:uncertain_function} is perturbed by this small amount $-\sum_{j=1}^{J} | B_{ij}| \eta_d$. The impact of this small perturbation on the leader objective value is $\sum_{ij} | B_{ij} y_i^{L^*} | \eta_d$, where $y_i^{L^*}$ is the optimal dual variable associated with each coordination constraint of the leader problem with the coordination variables $\by^F$ fixed to the values $\bar{y}^{F^*}$. If the estimates $\bar{y}^{F^*}$ are accurate, this also represents the dual variables of the leader problem $\cP^L(D^L,D^F_P,D^F_S)$ with the original sensitive data $D^F_S$.

By analogy, we can straightforwardly deduce that a lower-bound to $f(\bm{\epsilon})$ is given by $-\sum_{ij} | B_{ij} y_i^{L^*} | \eta_d$. Therefore, $|f(\bm{\epsilon})|$ (and the \emph{cost of privacy}) are bounded as follows:
\begin{equation}
\left|\hat{\theta}^{L^*} - \theta^{L^*}\right| \leq \eta_d \left|\left|B^T y^{L^*}\right|\right|_1.
\end{equation}
\end{proof}



While fidelity with respect to the follower objective value is
\emph{explicitly} enforced by Constraint \eqref{eq:pp3_1a}, the result
above ensures fidelity with respect to the leader objective value. This
fidelity is \emph{implicitly} enforced by Constraint \eqref{eq:pp3_1b}
on the follower coordination variables $\by^F$, and their impact
on the leader objective value via the coordination constraints 
\eqref{eq:coordination}. This sub-optimality in the leader cost
represents the so-called \emph{cost of privacy}.


\section{Experimental Evaluation}
\label{sec:experiments}

\def\Mp{\textsl{PPSM}_{primal}}
\def\Md{\textsl{PPSM}_{dual}}

The performance of the proposed PPSM is illustrated on the motivation
problem introduced in Section \ref{sec:preliminaries}. The leader
represents the GAUC problem, and the two followers represent the EM
and GM problems. The PPSM aims at preserving privacy on the sensitive
gas demand profiles ($D^g_S$) in the GM, and fidelity with respect to
the original Stackelberg game. Fidelity constraints are
\textit{explicitly} enforced on the original objective value of the GM
$(\cO^{g^*})$ and gas prices $(\by^{g^*})$. Additionally, the
\textit{implicit} impact of the PPSM on the original objective values
of the GAUC ($\cO^{uc^*}$) and the EM ($\cO^{e^*}$) is analyzed.

\smallskip\noindent\textbf{Case study setup}. 
The PPSM is evaluated on a test system that represents the joint
natural gas and electricity systems in the Northeastern US
\cite{byeon2019unit}. The system is composed of the IEEE 36-bus power
system \cite{allen2008combined} and a gas transmission network
covering the Pennsylvania-to-northeast New England area.  An in-depth
description of the GAUC problem and case-study setup will be provided
upon request.

This case study analyzes the performance of PPSM under various
operating conditions of the gas and electricity systems. The
electricity demand profile is uniformly increased by a stress factor
ranging from $30\%$ to $60\%$, and the gas demand profile is increased
by a stress factor ranging from $10\%$ to $130\%$, producing
increasingly stressed and difficult operating conditions.  The
experiments compare the proposed PPSM to a version (PPSM$_p$) that
omits the fidelity constraint on the dual variables
\eqref{eq:pp3_1b}. Both versions are compared with the standard
Laplace mechanism for varying values of the privacy parameter $\alpha
\in \{0.1, 1, 10\}\times 10^2$ MWh, and the fidelity parameters
$\eta_p = \eta_d \in \{0.01, 0.1, 10.0\}\%$ of the original objective
value of the GM $(\cO^{g^*})$ and gas prices $(\by^{g^*})$,
respectively. In the GAUC problem, the original, sensitive, gas demand 
vector is denoted $D^{g}_S$ (in lieu of $D^F_S$). Notice that, 
in the highest stress factor adopted,  this demand vector has minimum, 
median, and maximum values:  $0$, $3.38 \times 10^2$, $98.31 \times 10^2$, respectively. 
Therefore, the privacy parameters adopted ensure a very low privacy risk.

To induce uncertainty on the gas cost estimates $\bar{y}^g$ (in lieu of $\bar{y}^F$), 
the experiments use a noisy version resulted by perturbing the original estimates $\by^{g*}$ using Normal noise with standard deviation equal to $10\%$ of the average value of $\by^{g*}$. The experiments also evaluated benchmarks that use precise cost estimates. The results trends are consistent with those provided here.

We generate $30$ repetitions for each test case and report average results in all experiments and impose a 1-hour wall-clock limit.  A wall-clock limit of 1 hour is given to generate and solve each of the instances. The resolution of the privacy-preserving demand profiles (phases [1] and [2] of PPSM) takes less than 30s for any of the instances. 


\begin{table}[!t]
\centering
\resizebox{0.85\linewidth}{!}
{
\begin{tabular}{llrr @{\hspace{5pt}}|@{\hspace{5pt}} rrr}
\toprule
$\cM$ & $\alpha$ & sat.(\%) & $\Delta_{D^g_S}$ (L1) & $\Delta_{\cO^{uc}} (\%)$ &   $\Delta_{\cO^{e}} (\%)$&   $\Delta_{\cO^{g}} (\%)$ \\
\midrule
\multirow{3}{*}{Laplace}
       & 0.1   &  71.49 &   5.08    & 0.1237 &   0.3222 &   0.3335 \\
       & 1.0   &  18.13 &  50.85    & 1.2959 &   3.5538 &   3.5540 \\
       & 10.0  &   4.47 & 508.55    & 22.940 &   52.414 &  52.414 \\
\cline{2-3}
\cline{4-7}
\multirow{3}{*}{$\text{PPSM}_p$}
       & 0.1   & 98.45 &  4.45 &    0.0631 &   0.1503 &   0.1503 \\
       & 1.0   & 91.30 & 21.87 &    0.1216 &   0.1764 &   0.1761 \\
       & 10.0  & 80.71 & 24.31 &    0.2143 &   0.3851 &   0.3853 \\
\cline{2-3}
\cline{4-7}
\multirow{3}{*}{$\text{PPSM}$}
       & 0.1   & 99.10 &  3.89 &   0.0192 &   0.1056 &   0.1057 \\
       & 1.0   & 95.09 & 12.71 &   0.0698 &   0.1465 &   0.1465 \\
       & 10.0  & 91.35 & 14.16 &   0.1367 &   0.2330 &   0.2331 \\
\bottomrule
\end{tabular}
}
\caption{Left: Satisfactory instances (\%) and L1 errors (MWh) on the gas demands ($\Delta_{D^g_S}$) for varying indistinguishability parameters $\alpha$, and $\eta_p = \eta_d =0.1 \%$ of the leaders objective value and the dual variables values, respectively. 
Right: Errors (\%) on the leader objective ($\Delta_{\cO^{uc}}$) and followers' objectives ($\Delta_{\cO^e}$ and $\Delta_{\cO^g}$).
\label{tab:1}}
\end{table}

\subsubsection{Limits of The Laplace Mechanism}
\label{sec:laplace_limits}

This section studies the applicability of the Laplace mechanism to our context of interest. 
Table \ref{tab:1} (left) reports the percentage of the feasible solutions (over $1260$ instances) across different values of the privacy parameter $\alpha$. It compares the Laplace mechanism with $\text{PPSM}_p$ and $\text{PPSM}$. 
When $\alpha>0.1$ the Laplace-obfuscated gas demands rarely produce a feasible solution to the GAUC problem. 
\emph{These results justify the need for more advanced privacy-preserving mechanisms for Stackelberg games, and hence the proposed PPSM}. 
In contrast, the PPSMs result in a much higher number of feasible solutions. Indeed, all ``unsolved'' PPSM cases are due to timeout and not as a direct result of infeasibility.  
Additionally, we verified that the two PPSMs are always able to find a feasible solution to the GM follower problem with the gas demand profile $\hat{D}^g_S$.

Table \ref{tab:1} (left) also reports the L1 distance $\Delta_d$
between the original gas demand $D^g_S$ and the privacy-preserving
versions obtained with each of the mechanisms
analyzed. Unsurprisingly, the L1 errors increase with the increasing
of parameter $\alpha$, as larger values of $\alpha$ induce more noise.
However, the L1 errors introduced by the PPSM are much more contained
than the Laplace ones, producing more than an order of magnitude more
accurate results for the larger privacy parameters.  \emph{These
  results indicate that the highly-perturbed demand profiles induced
  by the Laplace mechanism lead to infeasibility in the GAUC problem,
  whereas both PPSM and PPSM$_p$ manage to restore
  \textit{feasibility} of the post-processed demand profiles.}

\subsubsection{Leader and Follower Objectives}

The next results evaluate the ability of PPSM to preserve the optimal objective values of the leader and the follower problems.  
Table \ref{tab:1} (right) tabulates the errors, in percentage, on the objective costs of the GAUC problem (leader), the EM problem and the GM problem (followers) at varying indistinguishability parameters $\alpha$, and for a fixed fidelity parameter $\eta =\eta_p=\eta_d = 0.1\%$.
The errors $\Delta_{\cO}$ are defined as $\frac{|\cO^* - \tilde{\cO^*}|}{\cO} \%$, where $\cO \in \{\cO^{uc}, \cO^{e}, \cO^{g}\}$ and are computed in expectation over the feasible instances only. 
Parameter $\alpha$ controls the amount of noise being added to the gas demand profiles, therefore, the objective costs are closer to their original values when $\alpha$ is small.
\emph{Observe that the PPSMs induce objective costs differences that are from one to two orders of magnitude more accurate than those induced by the Laplace mechanism, and that are at most $0.4\%$ of the original objective costs.}
Additionally, $\text{PPSM}$ is consistently more accurate than $\text{PPSM}_p$; By enforcing fidelity of the coordination variables $\by^g$, PPSM better limits the impact of the noise on the leader objective (GAUC), which in turn results in more faithful solutions for the followers' problems.

\begin{table}
\centering
\resizebox{0.7\linewidth}{!}
{
\begin{tabular}{llrrr}
\toprule
     $\cM$ & $\eta$ & $\Delta_{\cO^{uc}} (\%)$ &   $\Delta_{\cO^e} (\%)$&   $\Delta_{\cO^g} (\%)$ \\
\midrule
\multirow{1}{*}{Laplace}
     & NA &  22.9400 &  52.4141 &  52.4141 \\
\cline{2-5}
\multirow{3}{*}{$\textsl{PPSM}_{p}$}
     & 0.1\% &   0.1915 & 0.3851 & 0.3851\\
     & 1.0\% &   0.1946 & 0.4102 & 0.4102\\
     & 10.0\%&   0.2224 & 0.4543 & 0.4543\\
\cline{2-5}
\multirow{3}{*}{$\textsl{PPSM}$}
     & 0.1\% &   0.1367 & 0.2330 & 0.2330\\
     & 1.0\% &   0.1242 & 0.2060 & 0.2060\\
     & 10.0\% &  0.2086 & 0.3601 & 0.3605\\
     \bottomrule
\end{tabular}
}
\caption{Cost objective differences (\%) at varying fidelity parameters $\eta = \eta_p = \eta_d$ \%, and indistinguishability parameter $\alpha=10$.
\label{tab:2}}
\end{table}

\begin{figure*}[!ht]
\centering
\includegraphics[width=0.30\linewidth]{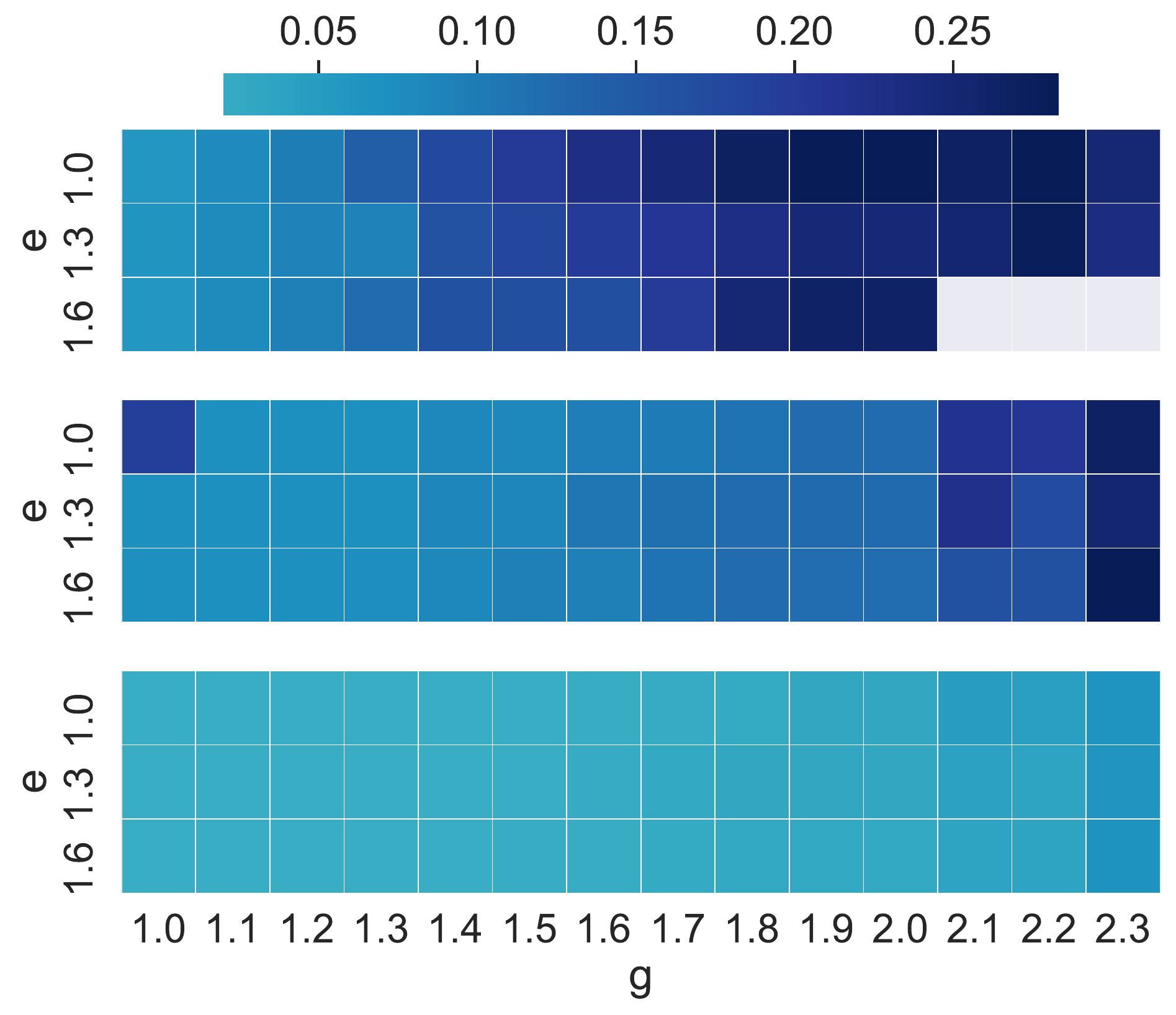}
\includegraphics[width=0.30\linewidth]{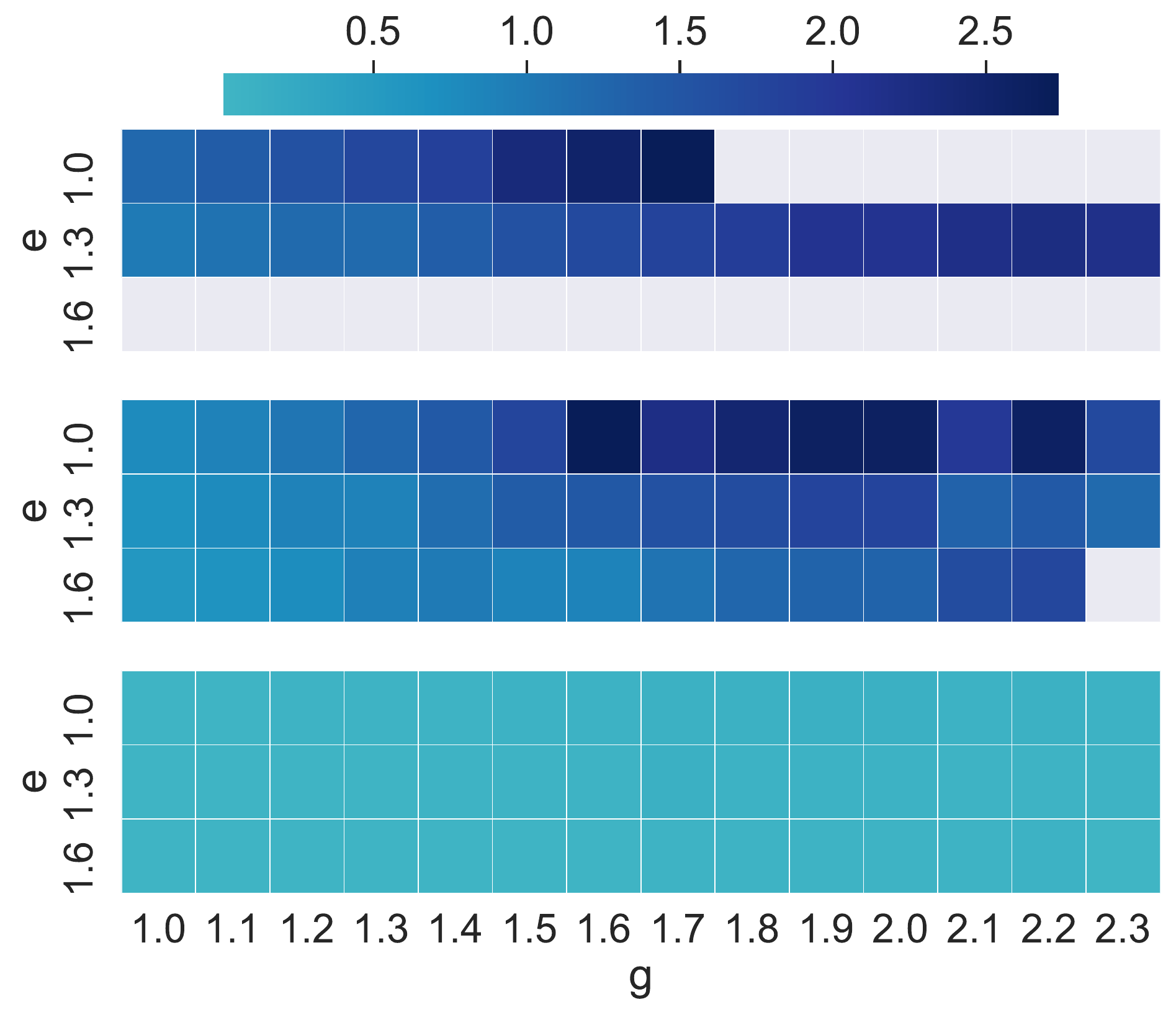}
\includegraphics[width=0.30\linewidth]{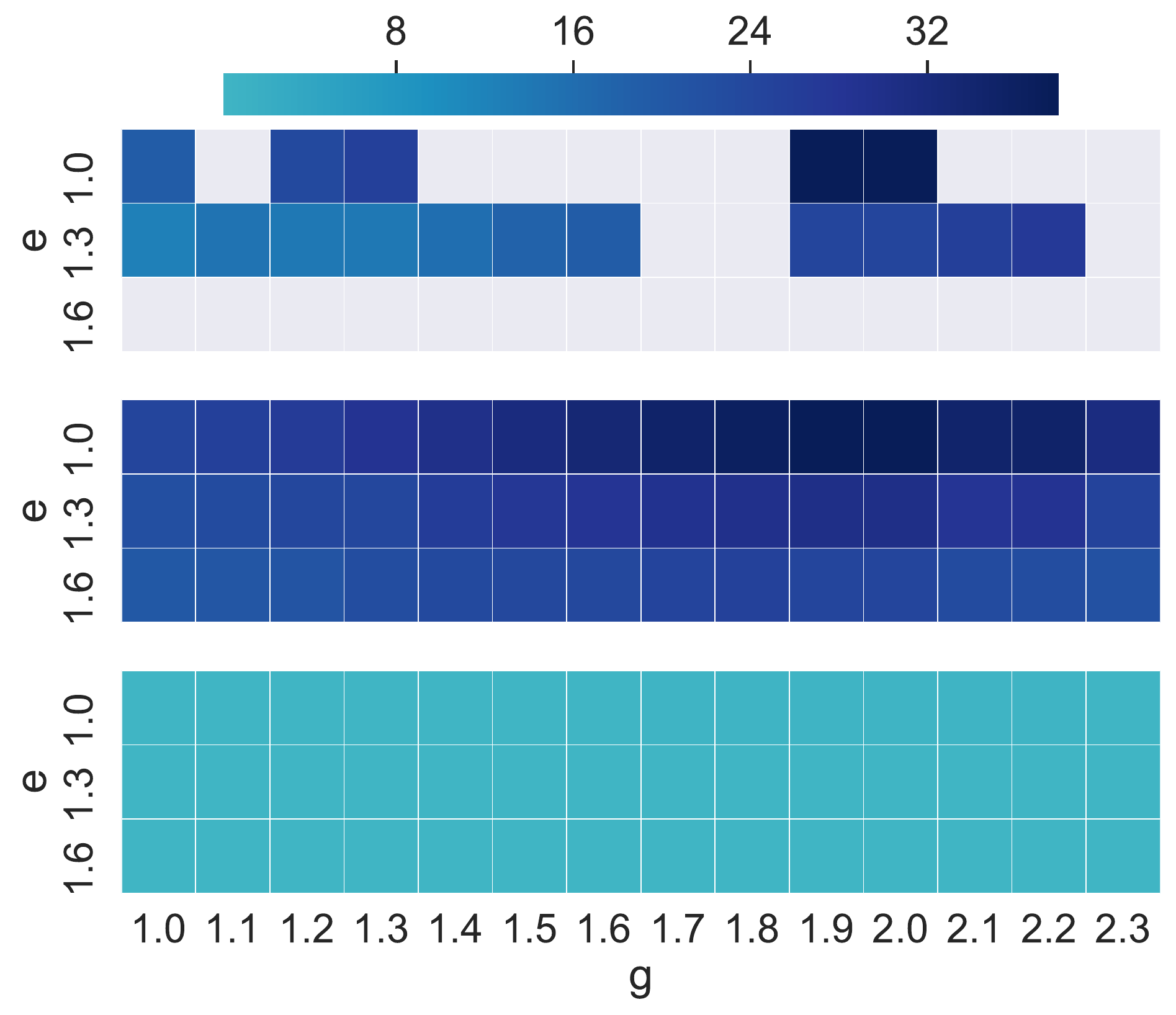}

\caption{GAUC objective cost difference (\%) at varying  gas (g) and electricity (e) stress levels for $\tilde{D}^g_S$ obtained via \emph{Laplace} (top), \emph{$\text{PPSM}_p$} (middle), \emph{$\text{PPSM}$} (bottom). Privacy: $\alpha\!=\!0.1$ (left), $1.0$ (center), $10.0$ (right). Fidelity: $\eta_p\!=\!\eta_d\!=\!0.   1 \%$ of respective quantities.
\label{fig:3}
}
\end{figure*}
These results are further illustrated in Table \ref{tab:2}, that analyzes the difference in objective costs at varying fidelity parameters $\eta_p$ and $\eta_d$, for a fixed privacy parameter $\alpha = 10$ (i.e., the largest privacy level attainable in our experimental setting). Once again, the results of the PPSM mechanisms are at least two orders of magnitude more precise than those obtained by the Laplace mechanism. Additionally, notice that the fidelity parameters $\eta_p$ and $\eta_d$ impact the accuracy of the privacy-preserving objective costs. Indeed, they indirectly control the deviation of the privacy-preserving GAUC and GM objectives with respect to the original counterparts. While the results differences are small, in percentage, their impact on the objective functions (in the $10^6$ order) is non-negligible.

\subsubsection{Stress Levels Analysis}
Finally, this subsection details the effect of the privacy-preserving mechanisms on the combined GAUC problem for all the electricity and gas stress levels adopted. 

Figure \ref{fig:3} reports heatmaps of the total (GAUC) objective cost difference, in percentage, at varying electricity (e) and gas (g) stress levels for the privacy-preserving data obtained via the Laplace mechanism (top), $\text{PPSM}_p$ (middle), and $\text{PPSM}$ (bottom). Each square represents the objective cost difference averaged over $30$ instances for a particular electricity and gas stress level. The darker the color, the more pronounced are the errors committed by the mechanisms, as reported in the legends on top of each subfigure. Gray squares represent the set of instances for which no feasible solution of the GAUC problem was found or when a timeout is reached. 
The illustration reports the cost differences for privacy parameters $\alpha\!=\!0.1$ (left) $\alpha\!=\!1.0$ (middle), and $\alpha\!=\!10.0$ (right).

These results illustrate three trends: 
Firstly, for every mechanism, the objective differences become more pronounced as the electricity and gas stress levels increase. This can be explained by the increased impact of the Laplace perturbations on higher values of gas demand profiles.
Secondly, they remark that the PPSMs produce privacy-preserving Stackelberg problems that are consistently more faithful to the original problems with respect to those produced by the Laplace mechanism.
Finally, they show that $\text{PPSM}$ is consistently more accurate than $\text{PPSM}_p$ across all stress levels. \emph{These results are significant, as they show the robustness of the proposed PPSMs over different electricity and natural gas demand profiles. They indicate that PPSM can provide a realistic and efficient solution for the coordination of electricity and natural gas markets.}

\section{Related Work}
\label{sec:related_work}

The obfuscation of data values under the lens of differential privacy is a challenging task that has been studied from several angles. Often, the released data is generated from a data synopsis in the form of a noisy histogram \cite{li2010optimizing,hay2016principled,qardaji:14,mohammed:11,xiao2010differentially}. 
These methods are typically adopted in the context of statistical queries. 
The design of markets for private data has also received considerable attention, see for instance \cite{ghosh2010selling,Niu:2018}.


However, all the proposals above, do not involve data used as input to a complex optimization problem, as in the case of this work. 
The closest work related to the proposal in this paper can be considered~\cite{Fioretto:18b,Fioretto:ijcai-19}, which, in the context energy networks, propose a privacy-preserving mechanism for releasing datasets that can be used as input to an \emph{optimal power flow} problem. A similar line of work uses hierarchical (bilevel) optimization for obfuscating the energy network parameters or locations while ensuring high utility for the problem of interest \cite{Fioretto:ijcai-19,Fioretto:TSG19}. 

In contrast to these studies, the proposed PPSM focuses on solving Stackelberg games in which the follower parameters are sensitive. PPSM also enforces the notion of fidelity of the privacy-preserving information with respect to the leader and follower objectives. Finally, to the best of our knowledge, this is the first DP mechanism that is applied to the coordination of sequential electricity and natural gas markets.  

\section{Conclusions}
\label{sec:conclusions}

This paper introduced a differentially private (DP) mechanism to protect the \emph{sensitive information} exchanged during the coordination of the sequential electricity and natural gas market clearings. The \emph{proposed Privacy-Preserving Stackelberg Mechanism (PPSM)} obfuscates the gas demand profile exchanged by the gas market, while also ensuring that the resulting problem preserves the fundamental properties of the original Stackelberg game. 
The PPSM was shown to enjoy strong properties: It complies with the notion of DP and ensures that the outcomes of the privacy-preserving Stackelberg mechanism are close-to-optimality for each agent.
Experimental results on several gas and electricity market benchmarks based on a real case study demonstrated the effectiveness of the approach: \emph{The PPSM was shown to obtain up to two orders of magnitude improvement on the cost of the agents when compared to a traditional Laplace mechanism}.


Future work will focus on several avenues, including extended theoretical bounds on the cost of privacy, studying the game-theoretic properties of this privacy-preserving Stackelberg game, accounting for uncertainty on the public data, and, studying the applicability of the PPSM to other domains.


\bibliographystyle{named}
\bibliography{bibliography,aamas}  


\appendix
\section{Gas-aware Electricity Unit Commitment Problem}

This section reviews the motivating application used in the paper. 

Due to the large penetration of Gas-Fired Power Plants (GFPPs) at the interface between electricity and natural gas systems, the operation of these energy systems is strongly interdependent. Yet, electricity and natural gas systems are operated in the day-ahead by sequential and independent markets. This lack of coordination may lead to significant economic inefficiencies and reliability risks in congested environments, as revealed by the 2014 polar vortex event in the northeastern United States. 

In order to improve the coordination between these systems while preserving the current organization of their respective markets, \cite{byeon2019unit} proposed to introduce gas network awareness into the current unit commitment problem. In this work, the coordination between electricity and natural gas markets is modeled as a Stackelberg game between a leader, i.e. the \emph{gas-aware electricity unit commitment} (GAUC), and two followers, i.e. the \emph{electricity market} (EM) and \emph{natural gas market} (GM).

\subsubsection*{\textbf{GAUC Coordination Variables}~($\bx^{uc}$)}
The GAUC is responsible for committing the bids of the electricity market participants, while anticipating the impact of these decisions on the follower, i.e. the sequential EM and GM clearings. Due to the participation of GFPPs at the interface between electricity and natural gas systems, the \emph{commitment} $\bx^{uc}$ of the selected bids of GFPPs by the GAUC impacts both the EM and GM clearings. 

Firstly, the EM uses the bids commitment $\bx^{uc}$ of the EM participants as input and dispatches these \textit{selected bids} to maximize the social welfare of the power system. Secondly, the GM dispatches its participants bids in order to maximize the social welfare in the natural gas system.

\subsubsection*{\textbf{EM and GM Coordination Variables}~($\bx^e$ and $\by^g$)}
The electricity dispatch $\bx^{e}\!\!$ of GFPPs is directly linked to their gas consumption $\bm{\gamma}^{g}\!\!$, which is used as input to the GM clearing problem. In return, the reaction of the follower, i.e., the electricity dispatch $\bx^{e}$ and gas prices $\by^{g}$, are accounted in the leader decisions through the \textit{objective cost} $\cO^{uc}$ and feasible space of the GAUC problem $\cF^{uc}$. 

Firstly, the GAUC problem objective accounts for the dispatch cost of the electricity bids dispatched $\bx^e$. 
Secondly, the GAUC problem feasible space accounts for \emph{bid-validity constraints} which embed the interdependencies between the gas prices $\by^g$ and the marginal 
electricity production cost of GFPPs.
In practice, these \emph{linear} coordination constraints enforce the price of the last selected GFPPs bid to be no larger than their marginal electricity production cost. These bidirectional interdependencies between the leader and the follower are characteristic of a classic Stackelberg game.

Note that, while, in this motivation problem, both the \emph{primal} and \emph{dual} coordination variables of the followers impact the leader problem, solely the dual variables $\by^g$ of the follower with sensitive data (GM) impact the leader feasible space. Therefore, this problem structure directly relates to the general Stackelberg game presented in the paper. 

\subsubsection*{\textbf{GAUC Problem}}
Therefore, the compact formulation of the GAUC as a hierarchical optimization problem is as follows:
\begin{subequations} \label{3level}
	\begin{alignat}{7}
	\cP^{uc} = 
	& \min_{\overset{\bx^{uc}\in\{0,1\}^N}{\underset{\bx^e ,  \by^g \geq \bm{0}}{}}} 
	&&  \;
	\cO^{\text{uc}} = 
	c^{uc^\top}\!\bx^{uc}  + c^{e^\top}\!\bx^e \label{3level0} \\
	& \quad \text{s.t.} && \bx^{uc} \in \cF^{uc} \label{3level1} \\
	& \quad                && A^{uc}  \bx^{uc}    +     B^{uc}  \by^g    \geq   \ b^{uc} \label{3level2} \\
	& \quad                &&  \bx^e , \by^g              \in      \text{primal and dual sol. of } \eqref{bilevel_heat_elec},  \label{3level3}
	\end{alignat}
\end{subequations}
where $N$ is the dimensionality of the commitment vector. 
The leader objective aims at minimizing the electricity system operating cost, which includes \emph{no-load} and \emph{start-up} costs $c^{uc^\top}\!\bx^{uc}$, and the \emph{cost of dispatching} the electricity bids $c^{e^\top}\!\bx^e$, constrained by the techno-economic characteristics of electricity suppliers \eqref{3level1}, and the bid-validity constraints \eqref{3level2}. 

\subsubsection*{\textbf{EM and GM Problems}}
Furthermore, due to the sequential order of EM and GM clearings, the two follower problems \eqref{3level3} can be expressed as a single hierarchical optimization problem, such that:
\begin{subequations} \label{bilevel_heat_elec}
\begin{alignat}{7}
\cP^{e} = 
	& \min_{\bx^e ,  \by^g}  \;
		&& \cO^{e} = c^{e^\top}\!\bx^e  
		\label{bilevel_heat_elec0} \\
		& \quad    \text{s.t.}   && \bx^e \in \cF^e  
		\label{bilevel_heat_elec1}  \\
		& \quad     &&  A^{e}  \bx^e + B^{e}    \bx^{uc} \geq b^{e}
		\label{bilevel_heat_elec1.2} \\
		& \quad  	&&   \by^g   \in   \text{dual sol. of }
	&&  \cP^{g} =  \min_{\bx^g} \;
		\cO^{g} = c^{g^\top}\!\bx^g 
		\tag{5a} \label{bilevel_heat_elec2}  \\
		& \quad 	&&  \quad   &&  \  
		\text{s.t. } \quad \bx^g  \in \cF^g \tag{5b}
		\label{bilevel_heat_elec4} \\
		& \quad 	&&  \quad   &&  \  
		\phantom{s.t. } \quad  A^{g}  \bx^g   +  B^{g}  
		\bx^e = d^{g}. \tag{5c}\label{bilevel_heat_elec5} 
\end{alignat}
\end{subequations}
\setcounter{equation}{5}
The objective \eqref{bilevel_heat_elec0} of the middle-level problem is to minimize the electricity dispatch cost, constrained by linearized power flow constraints \eqref{bilevel_heat_elec1} and bounds on the selected bids $\bx^{uc}$ of electricity suppliers \eqref{bilevel_heat_elec1.2}.
The lower-level problem \eqref{bilevel_heat_elec2}--\eqref{bilevel_heat_elec5} represents the GM clearing, which seeks to minimize the natural gas dispatch cost \eqref{bilevel_heat_elec2}, constrained by linearized gas flow constraints and the bounds on natural gas supply and demand bids \eqref{bilevel_heat_elec4}, as well as the nodal gas balance equation \eqref{bilevel_heat_elec5}.

The detailed expressions of these optimization problems, as well as the matrices $A^{uc}$, $B^{uc}$, $A^{e}$, $B^{e}$, $A^{g}$, $B^{g}$, the vectors $c^{uc}$, $b^{uc}$, $c^{e}$, $b^{e}$, $c^{g}$, $d^{g}$, and the polytopes $\cF^{uc}$, $\cF^{e}$, $\cF^{g}$ is derived from \cite{byeon2019unit}.

\subsubsection*{\textbf{Public and Sensitive Information}}
The GAUC problem takes as input the \emph{public} techno-economic characteristics of electricity 
suppliers, which are represented in Problem \eqref{3level} by the matrices $A^{uc}$ and $B^{uc}$, the vectors $c^{uc}$ and $b^{uc}$, and the polytope $\cF^{uc}$.
As the GAUC and EM cover the same energy system, these two agents also take the same parameters as input, which include the price $c^e$ and quantity $b^e$ bids of electricity suppliers, electricity demand profile, and electricity network technical characteristics. 
These parameters are represented in Problem \eqref{bilevel_heat_elec} by the matrices $A^{e}$ and $B^e$, the vectors $c^e$ and $b^e$, and the polytope $\cF^{e}$. 

Moreover, the GAUC takes as input the parameters of the the GM 
follower, which include the \emph{public} price $c^g$ and quantity bids of gas suppliers, gas network technical characteristics, as well as \emph{sensitive} gas demand profiles, denoted $D^g_S$. These parameters are represented in Problem \eqref{bilevel_heat_elec} by the matrices $A^g$ and $B^g$, the vectors $c^g$ and $D^F_S$, and the polytope $\cF^{g}$. 
In particular, the gas demand profiles $D^g_S = \left[d_{j}^g:\forall j \in \mathcal{V} \right]$ at all nodes $j \in \mathcal{V}$ of the gas network, represent the sensitive information (referred to as $D^F_S$ in Problem \eqref{eq:hierarchical}) of the GM follower.

In contrast, the following outcomes of the GM clearing are traditionally considered \emph{publicly available} information: the original objective value $\cO^{g^*}$ and natural gas prices $\by^{g^*}$. The mechanism introduced next will leverage this public information to train the predictor $\cM^F$ (step [3a]) of the PPSM.

\end{document}